\documentclass[12pt,reqno]{amsart} 
\usepackage{amssymb,amscd,url}
\usepackage{xcolor}
\usepackage{unicode-math}
\usepackage{mathrsfs}
\usepackage{amsfonts}
\usepackage{dsfont}
\usepackage[margin=1.2in]{geometry}
\usepackage{tikz} 
\usetikzlibrary{arrows.meta, positioning}
\usepackage{colonequals}
\usepackage{tikz-cd}
\usepackage{minted}

\begin{document}

\allowdisplaybreaks


\title[polynomials with rational periodic points]
      {Density of Unicritical Polynomials over Local Fields with Rational Periodic Points}
\date{\today}
\author[A. M. Dietrich]{Anna M. Dietrich}
\address{Mathematics Department, Box 1917
         Brown University, Providence, RI 02912 USA}
\email{anna\_dietrich@brown.edu}





\hyphenation{ca-non-i-cal semi-abel-ian}


\newtheorem{theorem}{Theorem}[section]
\newtheorem{lemma}[theorem]{Lemma}
\newtheorem{sublemma}[theorem]{Sublemma}
\newtheorem{conjecture}[theorem]{Conjecture}
\newtheorem{proposition}[theorem]{Proposition}
\newtheorem{corollary}[theorem]{Corollary}
\newtheorem*{claim}{Claim}

\theoremstyle{definition}
\newtheorem*{definition}{Definition}
\newtheorem{example}[theorem]{Example}
\newtheorem{remark}[theorem]{Remark}
\newtheorem{question}[theorem]{Question}

\theoremstyle{remark}
\newtheorem*{acknowledgement}{Acknowledgements}

\numberwithin{equation}{section}


\newenvironment{notation}[0]{%
  \begin{list}%
    {}%
    {\setlength{\itemindent}{0pt}
     \setlength{\labelwidth}{4\parindent}
     \setlength{\labelsep}{\parindent}
     \setlength{\leftmargin}{5\parindent}
     \setlength{\itemsep}{0pt}
     }%
   }%
  {\end{list}}

\newenvironment{parts}[0]{%
  \begin{list}{}%
    {\setlength{\itemindent}{0pt}
     \setlength{\labelwidth}{1.5\parindent}
     \setlength{\labelsep}{.5\parindent}
     \setlength{\leftmargin}{2\parindent}
     \setlength{\itemsep}{0pt}
     }%
   }%
  {\end{list}}
\newcommand{\Part}[1]{\item[\upshape#1]}

\def\Case#1#2{%
\paragraph{\textbf{\boldmath Case #1: #2.}}\hfil\break\ignorespaces}

\renewcommand{\a}{\alpha}
\renewcommand{\b}{\beta}
\newcommand{\g}{\gamma}
\renewcommand{\d}{\delta}
\newcommand{\e}{\epsilon}
\newcommand{\f}{\varphi}
\newcommand{\bfphi}{{\boldsymbol{\f}}}
\renewcommand{\l}{\lambda}
\renewcommand{\k}{\kappa}
\newcommand{\lhat}{\hat\lambda}
\newcommand{\m}{\mu}
\newcommand{\bfmu}{{\boldsymbol{\mu}}}
\renewcommand{\o}{\omega}
\renewcommand{\r}{\rho}
\newcommand{\rbar}{{\bar\rho}}
\newcommand{\s}{\sigma}
\newcommand{\sbar}{{\bar\sigma}}
\renewcommand{\t}{\tau}
\newcommand{\z}{\zeta}

\newcommand{\D}{\Delta}
\newcommand{\G}{\Gamma}
\newcommand{\F}{\Phi}
\renewcommand{\L}{\Lambda}

\newcommand{\ga}{{\mathfrak{a}}}
\newcommand{\gb}{{\mathfrak{b}}}
\newcommand{\gn}{{\mathfrak{n}}}
\newcommand{\gp}{{\mathfrak{p}}}
\newcommand{\gP}{{\mathfrak{P}}}
\newcommand{\gq}{{\mathfrak{q}}}

\newcommand{\Abar}{{\bar A}}
\newcommand{\Ebar}{{\bar E}}
\newcommand{\kbar}{{\bar k}}
\newcommand{\Kbar}{{\bar K}}
\newcommand{\Pbar}{{\bar P}}
\newcommand{\Sbar}{{\bar S}}
\newcommand{\Tbar}{{\bar T}}

\newcommand{\Acal}{{\mathcal A}}
\newcommand{\Bcal}{{\mathcal B}}
\newcommand{\Ccal}{{\mathcal C}}
\newcommand{\Dcal}{{\mathcal D}}
\newcommand{\Ecal}{{\mathcal E}}
\newcommand{\Fcal}{{\mathcal F}}
\newcommand{\Gcal}{{\mathcal G}}
\newcommand{\Hcal}{{\mathcal H}}
\newcommand{\Ical}{{\mathcal I}}
\newcommand{\Jcal}{{\mathcal J}}
\newcommand{\Kcal}{{\mathcal K}}
\newcommand{\Lcal}{{\mathcal L}}
\newcommand{\Mcal}{{\mathcal M}}
\newcommand{\Ncal}{{\mathcal N}}
\newcommand{\Ocal}{{\mathcal O}}
\newcommand{\Pcal}{{\mathcal P}}
\newcommand{\Qcal}{{\mathcal Q}}
\newcommand{\Rcal}{{\mathcal R}}
\newcommand{\Scal}{{\mathcal S}}
\newcommand{\Tcal}{{\mathcal T}}
\newcommand{\Ucal}{{\mathcal U}}
\newcommand{\Vcal}{{\mathcal V}}
\newcommand{\Wcal}{{\mathcal W}}
\newcommand{\Xcal}{{\mathcal X}}
\newcommand{\Ycal}{{\mathcal Y}}
\newcommand{\Zcal}{{\mathcal Z}}

\renewcommand{\AA}{\mathbb{A}}
\newcommand{\BB}{\mathbb{B}}
\newcommand{\CC}{\mathbb{C}}
\newcommand{\FF}{\mathbb{F}}
\newcommand{\GG}{\mathbb{G}}
\newcommand{\NN}{\mathbb{N}}
\newcommand{\PP}{\mathbb{P}}
\newcommand{\QQ}{\mathbb{Q}}
\newcommand{\RR}{\mathbb{R}}
\newcommand{\ZZ}{\mathbb{Z}}

\newcommand{\bfa}{{\boldsymbol a}}
\newcommand{\bfb}{{\boldsymbol b}}
\newcommand{\bfc}{{\boldsymbol c}}
\newcommand{\bfd}{{\boldsymbol d}}
\newcommand{\bfe}{{\boldsymbol e}}
\newcommand{\bff}{{\boldsymbol f}}
\newcommand{\bfg}{{\boldsymbol g}}
\newcommand{\bfi}{{\boldsymbol i}}
\newcommand{\bfj}{{\boldsymbol j}}
\newcommand{\bfp}{{\boldsymbol p}}
\newcommand{\bfr}{{\boldsymbol r}}
\newcommand{\bfs}{{\boldsymbol s}}
\newcommand{\bft}{{\boldsymbol t}}
\newcommand{\bfu}{{\boldsymbol u}}
\newcommand{\bfv}{{\boldsymbol v}}
\newcommand{\bfw}{{\boldsymbol w}}
\newcommand{\bfx}{{\boldsymbol x}}
\newcommand{\bfy}{{\boldsymbol y}}
\newcommand{\bfz}{{\boldsymbol z}}
\newcommand{\bfA}{{\boldsymbol A}}
\newcommand{\bfF}{{\boldsymbol F}}
\newcommand{\bfB}{{\boldsymbol B}}
\newcommand{\bfD}{{\boldsymbol D}}
\newcommand{\bfG}{{\boldsymbol G}}
\newcommand{\bfI}{{\boldsymbol I}}
\newcommand{\bfM}{{\boldsymbol M}}
\newcommand{\bfP}{{\boldsymbol P}}
\newcommand{\bfzero}{{\boldsymbol{0}}}
\newcommand{\bfone}{{\boldsymbol{1}}}

\newcommand{\Aut}{\operatorname{Aut}}
\newcommand{\codim}{\operatorname{codim}}
\newcommand{\Crit}{\operatorname{Crit}}
\newcommand{\Disc}{\operatorname{Disc}}
\newcommand{\Div}{\operatorname{Div}}
\newcommand{\Dom}{\operatorname{Dom}}
\newcommand{\End}{\operatorname{End}}
\newcommand{\Fbar}{{\bar{F}}}
\newcommand{\Fix}{\operatorname{Fix}}
\newcommand{\Gal}{\operatorname{Gal}}
\newcommand{\GL}{\operatorname{GL}}
\newcommand{\Hom}{\operatorname{Hom}}
\newcommand{\Image}{\operatorname{Image}}
\newcommand{\Isom}{\operatorname{Isom}}
\newcommand{\hhat}{{\hat h}}
\newcommand{\Ker}{{\operatorname{ker}}}
\newcommand{\limstar}{\lim\nolimits^*}
\newcommand{\limstarn}{\lim_{\hidewidth n\to\infty\hidewidth}{\!}^*{\,}}
\newcommand{\Mat}{\operatorname{Mat}}
\newcommand{\maxplus}{\operatornamewithlimits{\textup{max}^{\scriptscriptstyle+}}}
\newcommand{\MOD}[1]{~(\textup{mod}~#1)}
\newcommand{\Mor}{\operatorname{Mor}}
\newcommand{\Moduli}{\mathcal{M}}
\newcommand{\Norm}{{\operatorname{\mathsf{N}}}}
\newcommand{\notdivide}{\nmid}
\newcommand{\normalsubgroup}{\triangleleft}
\newcommand{\NS}{\operatorname{NS}}
\newcommand{\onto}{\twoheadrightarrow}
\newcommand{\ord}{\operatorname{ord}}
\newcommand{\Orbit}{\mathcal{O}}
\newcommand{\Per}{\operatorname{Per}}
\newcommand{\PrePer}{\operatorname{PrePer}}
\newcommand{\PGL}{\operatorname{PGL}}
\newcommand{\Pic}{\operatorname{Pic}}
\newcommand{\Prob}{\operatorname{Prob}}
\newcommand{\Proj}{\operatorname{Proj}}
\newcommand{\Qbar}{{\bar{\QQ}}}
\newcommand{\rank}{\operatorname{rank}}
\newcommand{\Rat}{\operatorname{Rat}}
\newcommand{\Resultant}{\operatorname{Res}}
\renewcommand{\setminus}{\smallsetminus}
\newcommand{\sgn}{\operatorname{sgn}} 
\newcommand{\SL}{\operatorname{SL}}
\newcommand{\Span}{\operatorname{Span}}
\newcommand{\Spec}{\operatorname{Spec}}
\renewcommand{\ss}{\textup{ss}}
\newcommand{\stab}{\textup{stab}}
\newcommand{\Support}{\operatorname{Supp}}
\newcommand{\tors}{{\textup{tors}}}
\newcommand{\tr}{{\textup{tr}}} 
\newcommand{\Trace}{\operatorname{Trace}}
\newcommand{\trianglebin}{\mathbin{\triangle}} 
\newcommand{\UHP}{{\mathfrak{h}}}    
\newcommand{\<}{\langle}
\renewcommand{\>}{\rangle}

\newcommand{\pmodintext}[1]{~\textup{(mod}~#1\textup{)}}
\newcommand{\ds}{\displaystyle}
\newcommand{\longhookrightarrow}{\lhook\joinrel\longrightarrow}
\newcommand{\longonto}{\relbar\joinrel\twoheadrightarrow}
\newcommand{\SmallMatrix}[1]{%
  \left(\begin{smallmatrix} #1 \end{smallmatrix}\right)}


\begin{abstract}
We prove an asymptotic formula for the density of parameters $c\in \ZZ_p$ for which the polynomial $x^d+c$ admits a $\QQ_p$-rational $n$-cycle. In the cases where $d=2$ and $p\neq 2,3$, we compute these densities explicitly for $n=1,2,$ and $3$.    
\end{abstract}

\maketitle


\section{Introduction}
\label{section:introduction}

The central aim of arithmetic statistics, broadly speaking, is to count number-theoretic objects with certain prescribed properties. Popular work in the area has focused on topics such as number fields of fixed degree and discriminant \cite{Bhargava2010,Davenport-Heilbronn71,Ellenberg-Venkatesh2006,LemkeOliver-Thorne2022},
random polynomials whose roots lie in a given field \cite{BCFG2022, Evans2006, Kac43}, elliptic curves with specified torsion subgroups or isogenies \cite{Boggess-Sankar2024,Harron-Snowden2017, Molnar-Voight2023}, and more recently, preperiodic points in families of rational maps \cite{Olechnowicz26}.

In \cite{GRV2025}, Gajović, Radičević, and Verzobio calculate rational functions in $p$ that describe the probability that a random Weierstrass equation corresponds to an elliptic curve with a $\QQ_p$-rational $3$-torsion point. The goal of this paper is to do the same for unicritical polynomials with a point of period $n$. Formally, we study the density of the sets 
\[\mathcal{A}_{p,d,n}\colonequals \{c\in \ZZ_p:x^d + c \text{ admits a } \QQ_p\text{-rational point of formal period } n\}.\]
We focus on unicritical polynomials because they are the simplest nontrivial example of one-parameter families of maps. Throughout, we normalize the natural Haar measure on $\ZZ_p$ so that $\mu(\ZZ_p)=1$, allowing us to interpret densities as probabilities.    
\begin{theorem}\label{theorem:genresult}
    Fix $d\geq 2$. Let 
    \[\Delta_n\colonequals \frac{1}{n} \sum_{k|n}\mathcal{M}\left(\frac{n}{k}\right)d^k,\]
    where $\mathcal{M}$ is the M\"{o}bius function\footnote{This is unconventional notation for the M\"{o}bius function, which is usually denoted $\mu$. However, the usual notation would be confusing in this context. Also, note that $n\Delta_n$ is the degree of the $n$th dynatomic polynomial.}. Then, 
    \[\lim_{p\to \infty}\mu_{\ZZ_p}(\mathcal{A}_{p,d,n}) = 1-\sum_{k=0}^{\D_n}\sum_{i=0}^{\D_n - k} \left(\frac{(-1)^i}{i!k!}\right)\left(\frac{n-1}{n}\right)^k.\]
\end{theorem}
As a corollary to Theorem \ref{theorem:genresult}, for large enough $p$, we see that the measure of $\mathcal{A}_{p,d,n}$ goes to $0$ as $n$ approaches infinity. 

We are able to prove more specific densities for all $p>3$ in the case of quadratic polynomials with rational cycles of sizes $n= 1,2$, and $3$. For periods $1$ and $2$, the densities have particularly simple closed forms.
\begin{theorem}\label{theorem:fixedand2}
    Fix $p\geq 2$. Then,
    \[\mu_{\ZZ_p}(\mathcal{A}_{p,2,1}) = \mu_{\ZZ_p}(\mathcal{A}_{p,2,2}) = \begin{cases}
        \frac{1}{2}, &\text{ if } p = 2,\\[5pt]
        \frac{p}{2(p+1)}, & \text{ if } p \geq 3.
    \end{cases}\]
\end{theorem}

The case of period $3$ is more subtle. In Section \ref{section:quadresult}, we introduce two rational functions, $\f$ and $\psi$, whose composition governs the density of quadratic polynomials with a $\QQ_p$-rational $3$-cycle. The resulting density depends on the image size of $\psi\circ \f$ modulo $p$, together with correction terms arising from its ramification. We note that this differs from the behavior observed in \cite{BCF2021,BCFG2022,FHP2021,GRV2025}, for example, where the density is described by a fixed rational function in $p$. The fact that this is not the case here may therefore be somewhat surprising. 

Since the precise statement of the theorem requires additional background, we include  Table \ref{table-deg2-per3}, which has the computed densities for several small primes, as a preliminary illustration. We refer the reader to Section \ref{section:period3points} and Theorem \ref{theorem:per3} for the relevant definitions and full statement. 

\begin{table}[h]
    \centering
    \begin{tabular}{@{\vrule height2.8ex depth1.5ex width0pt}|c|c|}
    \hline
         $p$& $\ZZ_p$ Density \\
         \hline
        $101$ & $\frac{9274443}{35030234} \approx 0.265$\\
        \hline
        $103$ & $\frac{7650441}{28410902} \approx 0.269$\\
        \hline
        $107$ & $\frac{6126187}{22050774} \approx 0.278$\\
        \hline
        $109$ & $\frac{40151949}{142453190} \approx 0.282$\\
        \hline
        $113$ & $\frac{23092850}{82245129} \approx 0.281$\\
        \hline
        $127$ & $\frac{71709789}{262193024} \approx 0.273$ \\
        \hline
    \end{tabular}
    \vspace{5pt}
    
    \caption{The $p$-adic density $\mu(\mathcal{A}_{p,d,n})$ of quadratic polynomials with a $\QQ_p$-rational $3$-periodic point.}
    \label{table-deg2-per3}
\end{table}
The paper is organized as follows. Section \ref{section:background} reviews the necessary background and tools used throughout. In Section \ref{section:generalresult}, we apply the Chebotarev Density Theorem to prove Theorem \ref{theorem:genresult}. Finally, in Section \ref{section:quadresult} we specialize to the family $x^2 + c$ and prove the results about periods $1,2$ and $3$.
\begin{acknowledgement}
The author thanks Joe Silverman for suggesting the problem, and for many helpful conversations and guidance. 
\end{acknowledgement} 
\section{Background}
\label{section:background}
Let $f$ be an automorphism of the projective line. We denote by $f^n$ the $n$th iterate of $f$, with the convention that $f^0 = \text{id}$ and $f^{n+1} = f \circ f^n$. A point $\alpha \in \PP^1$ has period $n$ if $f^n(\alpha) = \alpha$. In this paper, we study families of polynomials whose roots are the points of formal period $n$ for a map $f$. We state the definition in the unicritical case for clarity and notational purposes. 
\begin{definition}
    Let $f_{d,c}(x) = x^d + c$. The $n$th dynatomic polynomial is 
    \[\Phi_{n,d}(c,x) \colonequals \prod_{k|n} (f_{d,c}^k(x) - x)^{\mathcal{M}(n/k)},\]
    where $\mathcal{M}$ is the M\"{o}bius function\footnote{
    The fact that $\Phi_{n,d}(c,x)\in \ZZ[c,x]$ is explained in Chapter 4.2 of \cite{Silverman2007}, for example.}. 
\end{definition}

Some well-studied examples are the dynatomic polynomials for $x^2 + c$. The first few can be computed explicitly as 
\begin{align*}
\Phi_{1,2}(c,x) &= x^2 - x + c,\\
\Phi_{2,2}(c,x) &= x^2 + x + c + 1,\\
\Phi_{3,2}(c,x) &= x^6 + x^5 + (3c + 1)x^4 + (2c + 1)x^3 \\
&\quad + (3c^2 + 3c + 1)x^2 + (c^2 + 2c + 1)x  + (c^3 + 2c^2 + c + 1).
\end{align*}
We will use these explicit formulas extensively in Section \ref{section:quadresult}. The dynatomic modular curve $Y_1(n;d)$ is the affine curve cut out by \[\Phi_{n,d}(c,x)= 0.\] 
We denote the normalization of its projective closure by $X_1(n;d)$, and define $X_0(n;d)$ as the quotient of $X_1(n;d)$ by the subgroup of $\Aut(X_1(n;d))$ generated by $f_{d,c}(x)$. 

Irreducibility of $X_1(n;2)$ was shown by Bousch \cite{Bousch92}. This was later reproven and expanded to general $X_1(n;d)$ by Lau and Schleicher \cite{Lau-Schleicher94} and Morton \cite{Morton96}. For additional background on dynatomic modular curves, see Chapter $4$ of \cite{Silverman2007}. 

For fixed points and points of period $2$ of polynomials $x^2 + c$, we use $p$-adic integration to compute the relevant densities. With the appropriate setup, the theory of Igusa zeta functions significantly simplifies these calculations. In particular, we make use of the following change-of-variables formula for $p$-adic integrals due to Igusa \cite{Igusa2000} (Proposition 7.4.1). The theorem holds in greater generality, but we state only the case needed here.

\begin{lemma}\label{lem:igusa}
    Let $f=(f_1,\ldots,f_n):\ZZ_p^n \to \ZZ_p^n$, with all $f_i$ being $\ZZ_p$-analytic around a point $a\in \ZZ_p^n$. In addition, suppose that 
    \[\frac{\partial(f_1,\ldots,f_n)}{\partial(x_1,\ldots x_n)}(a) \neq 0.\]
    Then there are neighborhoods $U$ and $V$ around $a$ and $y = f(a)$ respectively such that $f:U\to V$ is a bijection, and 
     \[dy = \left|\frac{\partial(f_1,\ldots,f_n)}{\partial(x_1,\ldots x_n)}\right|_pdx.\]
\end{lemma}
\section{An Asymptotic Formula}
\label{section:generalresult}
Let $k$ be a number field. Suppose that $f(x)\in \mathcal{O}_k[x]$ is irreducible over $k$, and let $G$ be the Galois group of its splitting field over $k$. To study when $f$ has a root in $k_{\mathfrak{p}}$, we first study its reduction modulo $\mathfrak{p}$. For primes $\mathfrak{p}$ not dividing the discriminant of $f$, the reduction $\overline{f} \in (\mathcal{O}_k/\mathfrak{p})[x]$ is separable. Therefore, by Hensel's lemma, $\overline{f}$ has a root in $\mathcal{O}_k/\mathfrak{p}$ if and only if $f$ has a root
in $k_{\mathfrak{p}}$.

Thus, for all but finitely many primes, determining whether $f$ has a root
in $k_{\mathfrak{p}}$ is equivalent to determining whether $\overline{f}$ has a root modulo $\mathfrak{p}$.
The latter condition can be analyzed using the Chebotarev Density Theorem.
 
\begin{lemma}[Chebotarev Density Theorem]\label{lem:Chebotarev}
    Let $L/k$ be a Galois extension of global fields, and let $C\subset G = \text{Gal}(L/k)$ be a conjugacy class. Then, the density of the set
    \[\{\gp : \gp \text{ is a prime of } k, \gp\nmid \D_{L/k}, \text{Frob}_\gp \in C\}\]
    is $\#C / \#G$.
\end{lemma}
For a proof of Lemma \ref{lem:Chebotarev}, see for example \cite{Neukirch99} (Theorem 13.4). 
Let $L$ be the splitting field of $f$; then the group $\mathrm{Gal}(L/k)$ acts
on the roots of $f$. The Frobenius element $\mathrm{Frob}_{\mathfrak{p}}$
therefore acts as a permutation of the roots, and this permutation encodes
the factorization of $f$ modulo $\mathfrak{p}$. In particular, $f$ has a root in $k_{\mathfrak{p}}$ if and only if
$\mathrm{Frob}_{\mathfrak{p}}$ fixes at least one root of $f$. Using the Chebotarev Density Theorem, the probability that $f$ has a root mod $\mathfrak{p}$ is 
\[\frac{\#\{g\in G: g \text{ has a fixed point}\}}{\#G}.\]
We consider $\Phi_{n,d}$ as an element of $\QQ(c)[x]$ and use Lemma \ref{lem:Chebotarev} to calculate the probability that it has a root over finite fields. 
\begin{lemma}\label{lem: dynatomic gal}
    As in Theorem \ref{theorem:genresult}, define 
    \[\Delta_n\colonequals \frac{1}{n} \sum_{k|n}\mathcal{M}\left(\frac{n}{k}\right)d^k.\]
    Then, $\Phi_{n,d}$ is irreducible, and its Galois group over $\QQ(c)$ is isomorphic to $C_n\wr S_{\D_n}$. 
\end{lemma}
\begin{remark}
    The notation $C_n\wr S_{\D_n}$ is the wreath product of $C_n$ with $S_{\D_n}$. A concrete description of this group is given in the proof of Lemma \ref{lem:nofixedpoints}.
\end{remark}
\begin{proof}
    This was conjectured by Morton and Patel \cite{Morton-Patel94}(Conjecture 2). It was proved for $d=2$ by Bousch \cite{Bousch92} using methods from complex dynamics. It was then proved separately for general $d$ using different methods by Lau and Schleicher \cite{Lau-Schleicher94} and Morton \cite{Morton96}. 
\end{proof}

To apply the Chebotarev Density Theorem, we calculate the probability that an element of the Galois group $C_n\wr S_{\D_n}$ has a fixed point. 

\begin{lemma}\label{lem:nofixedpoints}
    Let $n\geq 2, m\geq 1$. The probability that an element of $C_n \wr S_m$ acting on $nm$ points leaves no point fixed is
    \begin{align}\label{eq:derangements}
        \sum_{k=0}^m \sum_{i=0}^{m-k} \left(\frac{(-1)^i}{i!k!}\right)\left(\frac{n-1}{n}\right)^k.
    \end{align}
\end{lemma}

\begin{proof}
    We model the action of $C_n \wr S_m$ on the $nm$ points by arranging the points into a collection of $m$ polygons, each with $n$ vertices. The group 
    $S_m$ acts by permuting the $n$-gons, and each copy of $C_n$ rotates a single $n$-gon. See Figure \ref{fig: n-gons} for an example.

    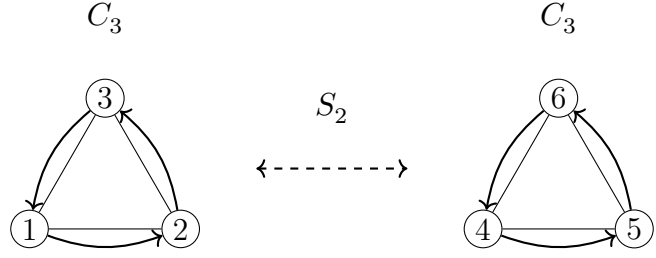
\begin{figure}
        \centering
        \begin{tikzpicture}[scale=2, every node/.style={circle, draw, inner sep=1.5pt}]

\node (A1) at (0,0) {1};
\node (A2) at (1,0) {2};
\node (A3) at (0.5,0.866) {3};

\node (B1) at (3,0) {4};
\node (B2) at (4,0) {5};
\node (B3) at (3.5,0.866) {6};

\draw (A1) -- (A2) -- (A3) -- (A1);
\draw (B1) -- (B2) -- (B3) -- (B1);

\draw[->, thick] (A1) to[bend right=20] (A2);
\draw[->, thick] (A2) to[bend right=20] (A3);
\draw[->, thick] (A3) to[bend right=20] (A1);

\draw[->, thick] (B1) to[bend right=20] (B2);
\draw[->, thick] (B2) to[bend right=20] (B3);
\draw[->, thick] (B3) to[bend right=20] (B1);

\node[draw=none] at (0.5,1.4) {$C_3$};
\node[draw=none] at (3.5,1.4) {$C_3$};

\draw[<->, thick, dashed] (1.5,0.4) -- (2.5,0.4);
\node[draw=none] at (2,0.8) {$S_2$};

\end{tikzpicture}
        \caption{The action of $C_3\wr S_2$ on $2$ triangles}
        \label{fig: n-gons}
    \end{figure}
    
    We first count the number of permutations in $S_m$ that fix exactly $k$ of the $n$-gons. Let $D_{\ell}$ denote the number of derangements of $\ell$ elements, i.e. the number of permutations in $S_{\ell}$ that have no fixed points. After choosing which $k$ of the $n$-gons remain fixed, the remaining $m-k$ polygons must be permuted with no fixed points. That is, the number of such permutations is 
    \[\binom{m}{k}D_{m-k}.\]
    For each of the $k$ fixed $n$-gons, if no point is fixed, then the corresponding element of $C_n$ must act by a nontrivial rotation. Only the identity element of a cyclic group fixes any point. Therefore, the probability that there is no fixed point given that there are $k$ fixed $n$-gons is
    \[\left(\frac{n-1}{n}\right)^k.\]
    Summing over all values $0\leq k\leq m$, the probability of choosing a permutation in $C_n\wr S_m$ with no fixed points is 
    \[\sum_{k=0}^m \binom{m}{k} \left(\frac{D_{m-k}}{m!}\right)\left(\frac{n-1}{n}\right)^k.\]
    Using a well-known formula for the number of derangements (see e.g. \cite{Stanley2011} Example 2.2.1), we can write 
    \[D_\ell = \ell! \sum_{i=0}^\ell \frac{(-1)^i}{i!}.\]
    Replacing $D_{m-k}$ with this identity gives the desired result. 
\end{proof}

\begin{proof}[Proof of Theorem \ref{theorem:genresult}]
    We would like to compute the probability that $\Phi_{n,d}$ has a root mod $p$. By Lemma \ref{lem: dynatomic gal}, the Galois group of $\Phi_{n,d}$ over $\QQ(c)$ is isomorphic to $C_n\wr S_{\Delta_n}$. Hilbert's irreducibility theorem (see \cite{Lang83} Chapter 9) guarantees that, outside of a thin set, all specializations $c \mapsto c_0$ preserve this Galois group. Applying Lemma \ref{lem:Chebotarev}, the probability that $\Phi_{n,d}$ has a root mod $p$ can be computed by calculating the probability of selecting an element of the Galois group that has a fixed point. Lemma \ref{lem:nofixedpoints} gives this probability as
    \[1-\sum_{k=0}^{\D_n}\sum_{i=0}^{\D_n - k} \left(\frac{(-1)^i}{i!k!}\right)\left(\frac{n-1}{n}\right)^k.\qedhere\]
\end{proof}

With this probability in hand, we see that as the cycle size $n$ tends to infinity, the probability that a map $x^d+c$ admits a $\QQ_p$-rational point of period $n$ goes to $0$. 
\begin{lemma}\label{cor:convergencerate}
    The sum \eqref{eq:derangements} satisfies the asymptotic formula
      \[
      \sum_{k=0}^m \sum_{i=0}^{m-k} \left(\frac{(-1)^i}{i!k!}\right)\left(\frac{n-1}{n}\right)^k
      = e^{-1/n} + O\left(\frac{1}{(m+1)! n^{m+1}}\right).
      \]
      In particular, it converges to~$1$ as $m,n\to\infty$.
\end{lemma}
\begin{proof}
    Let $j= i+k$. Then,
    \[\sum_{k=0}^m \sum_{i=0}^{m-k} \left(\frac{(-1)^i}{i!k!}\right)\left(\frac{n-1}{n}\right)^k = \sum_{j=0}^m \sum_{k=0}^j \left(\frac{(-1)^{j-k}}{(j-k)!k!}\right)\left(\frac{n-1}{n}\right)^k.\]
    Using the Binomial Theorem, 
    \begin{align}\label{eq: convergence series}
        \sum_{j=0}^m \left(\frac{(-1)^j}{j!}\right)\sum_{k=0}^j \binom{j}{k}\left(-\frac{n-1}{n}\right)^k &= \sum_{j=0}^m \left(\frac{(-1)^j}{j!}\right)\left(1 - \frac{n-1}{n}\right)^j&\nonumber\\
        &=\sum_{j=0}^m \frac{(-1/n)^j}{j!}.\nonumber&
    \end{align}
    Note that 
    \[e^{-1/n} = \sum_{j=0}^m \frac{(-1/n)^j}{j!} + \sum_{j=m+1}^\infty \frac{(-1/n)^j}{j!} = \sum_{j=0}^m \frac{(-1/n)^j}{j!} + O\left(\frac{1}{(m+1)!n^{m+1}}\right).\]
    Rearranging the equality gives the result. 
\end{proof}

\begin{example}\label{ex:period3density}
Let $n = 3$. The third dynatomic polynomial for $x^2 + c$ has degree $6$, so the asymptotic density can be calculated using the sum 
\[1- \sum_{k=0}^2\sum_{i=0}^{2-k} \left(\frac{2}{3}\right)^k\left(\frac{(-1)^i}{i!k!}\right) = 1-\frac{13}{18} = \frac{5}{18}.\]
\end{example}

We will see this value come up again in Section \ref{section:period3points}, when we use a parametrization of the dynatomic modular curve to calculate the density. While useful for all values of $n$, the main use of Theorem \ref{theorem:genresult} is to calculate approximate densities for $d$ and $n$ when the dynatomic curve is no longer genus $0$. 

\begin{example}
Table \ref{tab:approxden} shows the approximate asymptotic densities for the dynatomic curves associated to $x^2 + c$.

\begin{table}
    \centering
    \begin{tabular}{|c|c|c|}
        \hline
         $n$& Genus of Dynatomic Curve & Approximate Density \\
         \hline 
         $1$ & $0$ & $0.500$ \\
         \hline
        $2$ & $0$ & $0.500$\\
        \hline
        $3$ & $0$ & $0.278$\\
        \hline
        $4$ & $2$ & $0.221$\\
        \hline
        $5$ & $14$& $0.181$\\
        \hline
        $6$ &$34$ & $0.154$\\
        \hline
        $7$ & $134$ & $0.133$\\
        \hline
        $8$ &$285$ &$0.118$\\
        \hline
        $9$ &$745$ & $0.105$\\
        \hline
        $10$ & $1690$ & $0.095$\\
        \hline
    \end{tabular}
    \vspace{5pt}
    
    \caption{Approximate densities of polynomials $x^2 + c$ with a point of period $n$, with the genus of the corresponding dynatomic curve.}
    \label{tab:approxden}
\end{table}
\end{example}

\section{The case of quadratic polynomials}
\label{section:quadresult}
The transformation taking a polynomial of the form $Ax^2 + Bx + C$ to one of the form $x^2 + c$ preserves the dynamical properties of the original map. We therefore expect that studying densities of polynomials of the form $x^2 + c$ captures information about general quadratic polynomials with integral coefficients. The following lemma formalizes this expectation using integration. We first define three sets. Let
\begin{align*}
    S_1&\colonequals\{(A,B,C)\in \ZZ_p^3: Ax^2 + Bx + C \text{ has Property } \mathcal{P}\},&\\
    S_2&\colonequals\{(B,C)\in \ZZ_p^2: x^2 + Bx + C \text{ has Property } \mathcal{P}\}, \text{ and }&\\
    S_3&\colonequals\{C\in \ZZ_p: x^2 + C \text{ has Property } \mathcal{P}\},&
\end{align*}
where $\mathcal{P}$ is a dynamical property of quadratic polynomials that is invariant under $PGL_2$-transformation. 

\begin{lemma}\label{lem: measure invariance}
    Let $p>2$. The Haar measures of $S_1, S_2,$ and $S_3$ are equal, with respect to their ambient spaces. 
\end{lemma}
\begin{proof}
    There is a monic linear change of variable $x\mapsto x - B/2$ to map polynomials of the form $x^2 + Bx + C$ to polynomials of the form $x^2 + C$. Therefore, $\mu_{\ZZ_p^2}(S_2)  = \mu_{\ZZ_p}(S_3),$ since maps of this form are measure-preserving.

    To show that the measure of $S_1$ and $S_2$ are equal, define the intermediate sets
    \[S_1^i\colonequals \{(A,B,C)\in p^i\ZZ_p^* \times \ZZ_p^2: Ax^2 + Bx + C \text{ has Property } \mathcal{P}\}.\]
    There is a map $S_1^0 \to S_2$ with $x\mapsto A^{-1}x$. Using this map along with Lemma \ref{lem:igusa},
    \[\mu(S_1^0)= \int_{A\in \ZZ_p^*}\int_{(B,C)\in S_1^0} d\mu_{\ZZ_p^2}d\mu_{\ZZ_p} = \mu(S_2)\int_{A\in \ZZ_p^*}d\mu_{\ZZ_p}=\frac{p-1}{p}\mu(S_2).\]
    By properties of the Haar measure on $\ZZ_p$,
    \begin{align*}
        \mu_{\ZZ_p^3}(S_1) &= \int_{A\in \ZZ_p}\int_{(B,C)\in S_1} d\mu_{\ZZ_p^2}d\mu_{\ZZ_p}&\\
        &= \sum_{i=0}^\infty\int_{A\in p^i\ZZ_p^*}\int_{(B,C)\in S_1^i} d\mu_{\ZZ_p^2}d\mu_{\ZZ_p}&\\
        &= \sum_{i=0}^\infty \frac{1}{p^i}\int_{A\in \ZZ_p^*}\int_{(B,C)\in S_1^0} d\mu_{\ZZ_p^2}d\mu_{\ZZ_p}&\\
         &= \frac{p-1}{p}\mu(S_2)\sum_{i=0}^\infty \frac{1}{p^i}&\\
         &= \mu_{\ZZ_p^2}(S_2).&\qedhere
    \end{align*}
\end{proof}

For $n= 1,2,$ and $3$, the dynatomic modular curve is genus $0$, and we can use a parametrization of the curve to calculate the density explicitly.  

\subsection{Points of periods $1$ and $2$}
The first and second dynatomic polynomials are quadratic. Calculating the density of $c$ values with a $\QQ_p$-rational fixed point or point of period $2$ boils down to counting $c$ values for which these polynomials have a square discriminant. 

\begin{proof}[Proof of Theorem \ref{theorem:fixedand2}]
    We first calculate the density of polynomials with fixed points. The discriminant of the first dynatomic polynomial $x^2  - x + c$ is $1 - 4c$. Suppose that 
\[t^2 = 1 - 4c.\]
Then, the $c$ values with square discriminant satisfy $c = 4^{-1}(1 - t^2)$. Therefore, for $p>2$, we define a map $\ZZ_p \to \ZZ_p$ given by 
\begin{align}\label{eq:fixedmap}
    t \mapsto 4^{-1}(1 - t^2)
\end{align}
and calculate the measure of its image using $p$-adic integration. Suppose that $\mathscr{I}$ is the image of \eqref{eq:fixedmap}. Using the theory of Igusa zeta functions and applying Lemma \ref{lem:igusa}, 
\[\int_{\ZZ_p} \mathds{1}_{\mathscr{I}} d\mu = \frac{1}{2} \int_{\ZZ_p} |t|_pd\mu = \frac{p}{2(p+1)}.\]
When $p = 2$, the element $4$ is not invertible, and so we do the computation by hand. First observe that $1 - 4c$ is a unit in $\ZZ_2$, and that it is a square if and only if \begin{align}\label{eq:squaresinZ2}
    1-4c \equiv 1 \pmod 8.
\end{align}
Suppose that \eqref{eq:squaresinZ2} holds. This is true if and only if $c \equiv 0 \pmod 2$. Therefore, $1-4c$ is a square in $\ZZ_2$ if and only if $c\equiv 0 \pmod 2$, which holds for $1/2$ the values $c\in \ZZ_2$. 

A similar argument yields the density for points of period $2$. The second dynatomic polynomial is $x^2 + x + c+ 1$, with discriminant $-3 - 4c$. Therefore, the map of interest is 
\[t \mapsto 4^{-1}(-3 - t^2),\]
which has the same image density as the fixed point case.
\end{proof}
\subsection{Points of period $3$}\label{section:period3points}
We recall that the third dynatomic polynomial for $x^2 + c$ is 
\[
\begin{aligned}
\Phi_{3,2}(c,x) ={}& x^6 + x^5 + (3c + 1)x^4 + (2c + 1)x^3 \\
&+ (3c^2 + 3c + 1)x^2 + (c^2 + 2c + 1)x  + (c^3 + 2c^2 + c + 1).
\end{aligned}
\]
In the previous cases, we used the fact that the first and second dynatomic polynomials are quadratic, which allowed us to compute the densities via integration. However, once $n\geq 3$, the degree of the dynatomic polynomial begins to increase. In particular, although the third dynatomic curve still has genus $0$, the parametrization is given by a rational map instead of a morphism, which complicates the integration due to the presence of singularities. For points of period $3$, we use a different approach to avoid these singularities. 
\begin{proposition}
Let $k$ be a field of characteristic $0$, and consider the ring \[R\colonequals \frac{k[c,x]}{(\Phi_{3,2}(c,x))}.\] Define
\[ \text{$\a(c,x) \colonequals c$ and $\b(c,x) \colonequals x^4 + 2cx^2 + x^2 + x$.} \]
Further define an action $\s:R \to R$ by the rule
 \[ \s(c)=c\quad\text{and}\quad \s(x)=x^2+c,\]
so for $h(c,x) \in R$, we have $\sigma(h)(c,x) = h(c, x^2 + c)$. Let $R^\s$ be the subring of $R$ invariant under the action of $\s$. The polynomials $\a$ and $\b$ generate $R^\sigma$. Furthermore, $\a$ and $\b$ satisfy the relation
\begin{align}\label{eq:cyclecurve1}
    \a^4 + 4\a^3 + 2\a^2 \b+ 5\a^2 + 4\a\b + 3\a + \b^2 + \b + 2 = 0.
\end{align} 
\end{proposition}
\begin{proof}
    Consider the ring $k[\a,\b]$. Define a map 
    \[g : k[\a,\b] \to k[c,x]/(\Phi_{3,2}),\quad g(\a)=c,\quad g(\b)=x^4 + 2cx^2 + x^2 + x. \]
    We first show that this map is surjective onto $R^\s$. Because $\Phi_{3,2}(c,x)$ is monic of degree $6$ in $x$, the ring $R$ is a free module of rank $6$ over $k[c]$. Note that $\s$ is $k[c]$-linear. Consider the basis $\{1,x, x^2,x^3,x^4,x^5\}$ for $R$. We write the matrix associated to $\s$ using this basis:
    \[\begin{bmatrix}
        1 & c& c^2 & -2c^2 - c - 1 & -2c^3 - c^2 - c&3c^3 + c^2 + 2c\\
        0&0      &             0     & -c^2 - 2c - 1        &           1& 2c^3 + 4c^2 + 3c\\
         0         &          1          &       2c      &      -3c - 1     &         -4c^2   &        5c^2 + c\\
         0        &           0      &             0    &        -2c - 1       &            0 &    4c^2 + 2c + 1\\
          0        &           0       &            1      &            -1       &         -2c  &               2c\\
          0         &          0       &            0          &        -1      &             0             &    2c
    \end{bmatrix}.\]
    The polynomials in $R^\sigma$ satisfy $\sigma(h) = h$. Because $R$ is a free $k[c]$-module, we may extend scalars to $k(c)$ and compute generators for this eigenspace using row reduction. This calculation 
    shows that the eigenspace with eigenvalue $1$ is the $k(c)$-span of  
    $\{1, \b(c,x)\}$. Let \[M \colonequals \Span_{k[c]}\{1, \beta(c,x)\}.\] The $k[c]$-modules $M$ and $R^\sigma$ satisfy $M\subseteq R^\sigma\subseteq R$.
    We will now show that in fact $M = R^\sigma$. 
    
    Consider the quotient $R^\sigma/M$. Since $M\subseteq R^\sigma$, we can write the exact sequence
    \[\begin{tikzcd}
0 \arrow[r] & M \arrow[r] & R^\sigma \arrow[r] & R^\sigma/M\arrow[r] & 0.
\end{tikzcd}\]
Because $k(c)$ is the fraction field of $k[c]$, tensoring is exact. After extending scalars to $k(c)$, the induced map $k(c)\otimes_{k[c]} M \to k(c)\otimes_{k[c]} R^\sigma$ is an isomorphism by the row reduction computation, and thus $k(c)\otimes_{k[c]} R^\sigma/M = 0$. 
It follows that $R^\sigma/M$ is torsion. On the other hand, since $\beta(c,x) = x^4 + (2c+1)x^2 + x$, a change of basis for $R$ gives
\[R = \Span_{k[c]}\{1, \beta, x^2, x^3, x^4, x^5\} \cong M \oplus \Span_{k[c]}\{x^2, x^3, x^4, x^5\}.\]
That is, $R/M\cong k[c]^4$ as $k[c]$-modules, and so $R/M$ is torsion-free. All submodules of torsion-free modules are torsion-free, so $R^\sigma/M$ is also torsion-free. This forces $R^\sigma/M$ to be $0$. We can conclude from this that $R^\sigma$ is contained in the image of $g$.

    Finally, we calculate $\ker(g)$. Since $R$ is an integral domain, then $\ker(g)$ is a prime ideal in $k[\a,\b]$. The map $\sigma$ has finite order, so it generates a finite group acting on $R$, and therefore $R$ is integral over $R^\sigma$. Because Krull dimension is preserved under integral extension, we have 
    \[\dim\left(\frac{k[c,x]}{(\Phi_{3,2})}\right)^\sigma = \dim \frac{k[c,x]}{(\Phi_{3,2})} = 1.\]
     Then, by the first isomorphism theorem, $k[\a,\b]/(\ker(g))$ is dimension $1$, implying that $\ker(g)$ is a height $1$ prime ideal. All height $1$ prime ideals in a unique factorization domain are principal. Therefore, a single irreducible element in $\ker(g)$ can be chosen as the generator. 
    
    Direct computation verifies that $\a(c,x)$ and $\b(c,x)$ satisfy \eqref{eq:cyclecurve1}, so it is in $\ker(g)$ by construction. Consider \eqref{eq:cyclecurve1} in the ring $k(\a)[\b]$. Since \eqref{eq:cyclecurve1} is quadratic in $\b$, then it is reducible if and only if its discriminant is a square in $k(\a)$. However, direct computation shows that the discriminant and its derivative have greatest common divisor equal to $1$, so it is irreducible. That is, \eqref{eq:cyclecurve1} is a generator for $\ker(g)$.       
\end{proof}

\begin{remark}
    The curve \eqref{eq:cyclecurve1} is an explicit affine model for $X_0(3;2)$.
\end{remark}
\begin{lemma}
    The curve \eqref{eq:cyclecurve1} has genus $0$ and can be parametrized by
    \[t\mapsto(\a(t),\b(t)) = \left(-\frac{2t^2 + 7t + 7}{(t+2)^2},\frac{t^3 + 3t^2 + 2t-1}{(t+2)^4}\right).\]
\end{lemma}
\begin{proof}
    One can verify with a computer algebra system that $\a(t)$ and $\b(t)$ satisfy \eqref{eq:cyclecurve1}. Furthermore, there is a rational inverse
    \[(\a, \b)\mapsto \frac{\a^2 + \b - 3}{\a + 2},\]
    which can also be verified by direct computation.
\end{proof}

With this parametrization, finding rational points on \eqref{eq:cyclecurve1} is straightforward. However, rational points on this curve do not correspond to maps $x^2 + c$ with a rational point of period $3$; instead, they correspond to maps $x^2 + c$ admitting a rational $3$-cycle in a weaker sense. That is, the $3$-cycle determined by $(\a,\b)$ on this curve may not consist of rational points. For example, when $t = 0,$ we get the polynomial $x^2 - 7/4$, which does not have rational $3$-periodic points, as can be seen from the factorization of the dynatomic polynomial
\[\Phi_{3,2}(-7/4, x) = \frac{1}{64}(8x^3 - 4x^2 + 18x + 1)^2.\]
This raises the following question: how do we determine which points on \eqref{eq:cyclecurve1} lift? Suppose we have a solution $(\a(t),\b(t))$ to \eqref{eq:cyclecurve1} that is the image of a rational root under the quotient map. Then, the equation
\[x^4 + 2\a(t)x^2 + x^2 + x - \b(t) = 0,\]
coming from the definition of $\b$, necessarily has rational solutions. Factoring this expression, we require rational solutions to 
\[\left(x - \frac{1}{t+2}\right)\left(x^3 + \frac{1}{t+2}x^2 - \frac{3t^2 + 10t + 9}{(t+2)^2}x + \frac{t^3 + 3t^2 + 2t-1}{(t+2)^3}\right) = 0.\]

\begin{remark}
    Note the extraneous root $x = 1/(t+2)$. This arises as the sum of the roots of the cubic and is an artifact of the way $\b$ is calculated. Occasionally, this extra root corresponds to a preperiodic point leading in to the cycle. For example, when $t= -2/3$, we get the preperiodic point $3/4$ for the map $x^2 - 29/16$. However, this behavior is not consistent in general. The relevant points in the $3$-cycle are the roots of the cubic.
\end{remark} 

It turns out that after clearing denominators, the curve 
\begin{align}\label{eq:cyclecurve2}
    (t+2)^3x^3 + (t+2)^2x^2 - (t+2)(3t^2 + 10t + 9)x + t^3 + 3t^2 + 2t-1 = 0
\end{align}
also has genus $0$.
\begin{lemma}
    The curve \eqref{eq:cyclecurve2} has genus $0$ and can be parametrized by
    \[s \mapsto (t(s), x(s)) = \left(-\frac{2(s^3 + 15s^2 + 12s - 1)}{3(s^3 + 6s^2 + 3s - 1)}, -\frac{s^3 + 12s^2 + 9s + 5}{2(s-1)(s+2)(2s+1)}\right).\]
\end{lemma}
\begin{proof}
    A computer algebra system can verify that the given parametrization satisfies \eqref{eq:cyclecurve2}. It has rational inverse 
    \[(t,x)\mapsto \frac{-(6t^2x^3 - 18t^2x + 6t^2 +
  24tx^3 - 63tx + 24t +
  24x^3 - 60x + 25)}{(3t + 2)}.\qedhere\]
\end{proof}

Combining these two rational parametrizations, we compute the rational points on \eqref{eq:cyclecurve1} that lift to rational points above. We define the maps 
\begin{align*}
    \psi(t)&\colonequals -\frac{2t^2 + 7t + 7}{(t+2)^2},&\\[5pt]
    \varphi(s)&\colonequals -\frac{2(s^3 + 15s^2 + 12s - 1)}{3(s^3 + 6s^2 + 3s - 1)}.
\end{align*}
The image of the composition $\psi\circ\varphi$ gives $c$ values for which $x^2 + c$ has a rational $3$-periodic point.

We now analyze the size of the image of $\psi\circ \varphi$ on $\PP^1(\FF_p)$, which will give an approximation of the density of $c\in\ZZ_p$ with the desired property. Note that we exclude the primes $p = 2,3$ since $\psi\circ\varphi$ has bad reduction at these primes. 

\begin{lemma}\label{lem:phiimagesize}
    Let $p\geq 5$. Over $\PP^1(\FF_p)$, the density of the image of $\varphi$ is 
    \[\begin{cases}
        \frac{1}{3}, & \text{ if } p \equiv 2\pmod 3,\\[5pt]
        \frac{p+5}{3(p+1)}, &\text{ if } p \equiv 1\pmod 3.
    \end{cases}\]
\end{lemma}
\begin{proof}
    The critical points of $\varphi$ are the primitive cube roots of unity, as can be seen from its derivative
    \[\varphi'(s) = \frac{6(s^2 + s + 1)^2}{(s^3 + 6s^2 + 3s -1)^2}.\]
    The primitive cube roots of unity exist in $\FF_p$ only when $p\equiv 1 \pmod 3$. For $p\equiv 2\pmod 3$, this means that $\varphi$ acts as a true $3$ to $1$ map onto its image, so the density is 
    \[\frac{p+1}{3(p+1)} = \frac{1}{3}.\]
    For primes $p\equiv 1\pmod 3$, the two cube roots of unity map to two unique points with multiplicity, while the remaining points map $3$ to $1$. This gives
    \[2 + \frac{p-1}{3} = \frac{p+5}{3}\]
    points in the image. 
\end{proof}

Define an involution \[\iota(x) = \frac{-x}{x+1},\]
and observe that 
\[\psi(\iota(x)) = -\frac{2x^2 -7x(x+1) + 7(x+1)^2}{(2(x+1) - x)^2} = \psi(x).\]
That is, $\iota$ pairs up points that map to the same image under $\psi$.
\begin{definition}
    If $\varphi(s_1)$ and $\varphi(s_2)$ are related by $\iota(\varphi(s_1)) = \varphi(s_2)$, then $\varphi(s_1)$ and $\varphi(s_2)$ are in a \textit{collision pair}.
\end{definition}
We can describe the size of the image of $\psi\circ\varphi$ as the size of the image of $\varphi$ minus the number of collision pairs. We make this idea formal in Proposition \ref{prop:Fpdensity}. 

\begin{proposition}\label{prop:Fpdensity}
    Let $p\neq 2,3,7$. The size of the image of $\psi\circ\varphi$ over $\PP^1(\FF_p)$ is 
    \[I - \frac{N - 12B - 9K}{18},\]
    where 
     \[
      N \colonequals \#\bigl\{ (s_1,s_2)\in \FF_p^2 \mid
      \varphi(s_1)\varphi(s_2) + \varphi(s_1) + \varphi(s_2) = 0 \bigr\}, 
      \]
    $B$ is the number of critical values of $\varphi$ in a collision pair, and 
     \begin{align}\label{eq:KandI}
        K &\colonequals \begin{cases}
        2, &\text{ if there is } s\in \PP^1(\FF_p) \text{ with } \varphi(s)= 0,\\
        1, & \text{ otherwise},
        \end{cases} \\
        I &\colonequals \begin{cases}
        \frac{1}{3}, & \text{ if } p \equiv 2\pmod 3,\\[5pt]
        \frac{p+5}{3(p+1)}, &\text{ if } p \equiv 1\pmod 3.
        \end{cases}
      \end{align}
\end{proposition}
\begin{proof}
    As stated in the discussion before the proposition, the size of the image of $\psi\circ \varphi$ is the size of the image of $\varphi$ minus the number of collision pairs. By Lemma \ref{lem:phiimagesize}, the size of the image of $\varphi$ is $I$. That is, we only need to calculate the number of collision pairs.
    
    Suppose that $\varphi(s_1) = \iota(\varphi(s_2))$. This is equivalent to
    \begin{align}\label{eq:collisioncurve}
        \varphi(s_1)\varphi(s_2) + \varphi(s_1) + \varphi(s_2) = 0.
    \end{align} 
    Let $\{t,\iota(t)\}$ be a pair with $t\neq \iota(t)$, and consider their preimages under $\varphi$. There are $9$ pairs $(s_1,s_2)$ with $\varphi(s_1) = t$ and $\varphi(s_2) = \iota(t)$. Because \eqref{eq:collisioncurve} is symmetric in $s_1$ and $s_2$, then each pair $(s_1,s_2)$ contributes its reversal $(s_2,s_1)$ as well. That is, per unordered pair $\{t, \iota(t)\}$, there are $18$ solutions to \eqref{eq:collisioncurve}. 

    For the two fixed points of $\iota$, which are $t= 0$ and $t=-2$, this symmetry no longer produces distinct pairs, so each only contributes $9$ solutions. Because the total image size is changed by these fixed points, we would like to know how many show up in the image of $\varphi$. Calculating $\varphi(1)$ shows that $t=-2$ is always in the image of $\varphi$. With this in mind, define $K$ as in \eqref{eq:KandI}.
    
    When $p\equiv 1\pmod 3$, the critical points of $\varphi$ are defined over $\FF_p$. Each critical value has only a single preimage, so we also count the number of critical values $t$ with $\psi(t) = \psi(\iota(t))$. Let $B$ be the number of such points. For each $t$, there are exactly $6$ pairs $(s_1,s_2)$ satisfying \eqref{eq:collisioncurve}. 

    Let $N$ be the total number of pairs $(s_1,s_2)\in \FF_p^2$ satisfying \eqref{eq:collisioncurve}. Then, the number of points in the image of $\varphi$ mapping to the same image under $\psi$ is given by 
    \[\frac{N - 9K - 12B}{18}.\qedhere\]
\end{proof}

\begin{remark}
    The prime $7$ is excluded from Proposition \ref{prop:Fpdensity} because it has the unique property among primes that one of the critical points of $\varphi$ maps to $0$, a fixed point of the involution $\iota$. Indeed, if $\omega$ is a primitive cube root of unity, then 
    \[\varphi(\omega) = -\frac{2(\omega+5)}{3(\omega+2)}.\]
    If $\varphi(\omega) = 0$, then $\omega = -5$. Since $\omega$ is a cube root of unity in $\FF_p$,
    \[\omega^2 + \omega + 1 = 21 \equiv 0\pmod p,\]
    forcing $p\mid 21$. The only prime greater than $3$ with this property is $p = 7$. A similar argument shows that there are no primes with a primitive cube root mapping to $-2$, the other fixed point of $\iota$.
    
    This fundamentally changes the image count when $p = 7$, because $0$ is simultaneously a fixed point of $\iota$ and a critical value of $\varphi$. However, the prime $7$ is small enough that we can calculate by hand, yielding an image size of $4$. See Remark \ref{rmk:geometry} for another case when $p = 7$ is exceptional in this context.
\end{remark}

With the density approximation from Proposition \ref{prop:Fpdensity}, we can analyze the ramification of $\psi\circ \varphi$ to finish the density calculation. Intuitively, any obstruction to lifting the residue classes of the image to $\ZZ_p$ comes from the critical points of $\psi\circ\varphi$.  The derivative is
  \[(\psi\circ\varphi)'(x) = \frac{-27(x^2 + x + 1)^2(x^3 +15x^2 + 12x - 1)}{2(x-1)^3(x+2)^3(2x+1)^3}.\]
That is, potential obstructions come from the images of the roots of the polynomials $x^2 + x + 1$, $x^3 +15x^2 + 12x - 1$, and $\infty$. It turns out (and will be justified in the proof of Theorem \ref{theorem:per3}) that the roots of $x^3 +15x^2 + 12x - 1$ all map to the same residue class. Therefore, there are at most $4$ residue classes in the image that do not lift fully to $\ZZ_p$. As $p$ gets large, the contributions from these residue classes become negligible, so the asymptotic density will be the same as the density over finite fields. However, studying ramification will allow us to be more precise.    

\begin{lemma}\label{lem: ram points}
    The polynomial $x^3 +15x^2 + 12x - 1$ splits over $\FF_p$ if and only if $p\equiv \pm 1 \pmod 7$. Furthermore, if it splits, it splits completely.  
\end{lemma}
\begin{proof}
    Over $\QQ(\zeta_7)$, the polynomial splits completely, and has roots
    \begin{align*}
        3(\zeta_7^{-2} + \zeta_7^2) + 6(\zeta_7^{-3} +\zeta_7^3) -2, \\
        3(\zeta_7^{-2} + \zeta_7^2) - 3(\zeta_7^{-3} +\zeta_7^3) - 5,\\
        -6(\zeta_7^{-2} + \zeta_7^2) -3(\zeta_7^{-3} +\zeta_7^3) - 8.
    \end{align*}
    We can write these roots in terms of $\zeta_7 + \zeta_7^{-1}$ with the following identities:
    \begin{align*}
        \zeta_7^{-2} + \zeta_7^{2} & = (\zeta_7+ \zeta_7^{-1})^2 - 2,&\\
        \zeta_7^{-3} + \zeta_7^{3} & = (\zeta_7+\zeta_7^{-1})^3 - 3(\zeta_7+\zeta_7^{-1}).&
    \end{align*}
    Therefore, the splitting field of $x^3 +15x^2 + 12x - 1$ is a subfield of the totally real field $\QQ(\zeta_7)^+ = \QQ(\zeta_7 + \zeta_7^{-1})$. The degree of $\QQ(\zeta_7)^+$ over $\QQ$ is $3$. The polynomial $x^3 +15x^2 + 12x - 1$ is irreducible over $\QQ$, so the degree of its splitting field over $\QQ$ is also $3$. This forces the splitting field of the polynomial over $\QQ$ to be isomorphic to $\QQ(\zeta_7)^+$. 

    The Galois group of $\QQ(\zeta_7)^+$ over $\QQ$ is isomorphic to $(\ZZ/7\ZZ)^*/\{\pm1\}$. The polynomial splits completely $\pmod p$ if and only if the Frobenius element is trivial in $(\ZZ/7\ZZ)^*/\{\pm1\}$. The Frobenius class is determined by $p \pmod 7$ up to sign, and the identity coset is exactly $\{\pm1\}$. That is, 
    $x^3 +15x^2 + 12x - 1$ splits over $\FF_p$ if and only if $p\equiv \pm 1 \pmod 7$.
\end{proof}

\begin{theorem}\label{theorem:per3}
  Assume $p\neq 2,3$. Let $S_p$ be the image size of $\psi\circ\f$ in $\PP^1(\FF_p)$. Let $D$ be the number of residue classes in the image of $\psi\circ \varphi$ whose fiber consists of $3$ points each of ramification index $2$. Then,
  \[
    \mu_{\ZZ_p}(\mathcal{A}_{p,2,3}) = 
    \begin{cases}
      \dfrac{S_p - D}{p+1} + \dfrac{D(p+1)}{2p^3}, &\text{if $p\equiv 1 \pmod 3$}, \\[3.5\jot]
      \dfrac{S_p - D - 2}{p+1} + \dfrac{D(p+1)}{2p^3} + \dfrac{2(p+2)}{3p^4},
      &\text{if $p\equiv 2 \pmod 3$}. \\
    \end{cases}
    \]
\end{theorem}
\begin{proof}
    For now, assume that $x^2 + x + 1$ and $x^3 +15x^2 + 12x - 1$ split completely in $\FF_p$; the polynomial $x^2 + x+ 1$ splits if and only if $p\equiv 1\pmod 3$, and $x^3 +15x^2 + 12x - 1$ splits if and only if $p\equiv \pm1 \pmod 7$, as shown in Lemma \ref{lem: ram points}. 
    
    We start with the roots of $x^3 +15x^2 + 12x - 1$. Suppose its roots are $r_1, r_2, r_3 \in \FF_p$. We first show that 
    \[\psi\circ\varphi(r_1) = \psi\circ\varphi(r_2) = \psi\circ\varphi(r_3).\]
    Consider the polynomial $P(x)Q(r_1) - P(r_1)Q(x)$. The roots of this polynomial are the points mapping to $\psi\circ\varphi(r_1)$. From direct computation,
    \[P(x) = -\frac{7}{4}Q(x) \pmod{x^3 +15x^2 + 12x - 1}.\]
    Therefore, $P(x)Q(r_1) - P(r_1)Q(x)$ is $0$ mod $x^3 +15x^2 + 12x - 1$, so it has $r_1$, $r_2$, and $r_3$ as roots. That is, $r_1, r_2,$ and $r_3$ all map to a common image point $y$ under $\psi\circ\varphi$.  

    Let $D_{r_i}$ be the disk $r_i + p\ZZ_p$. From the derivative computation, each $r_i$ is a simple critical point of $\psi\circ\varphi$, so its ramification index is $2$. We can use this to write an expansion of $\psi\circ \varphi$ around $r_i$ for each $i$: 
    \[\psi\circ\varphi(x) = -\frac{7}{4} + c_i(x-r_i)^2 + O((x-r_i)^3),\]
    where $c_i = (\psi\circ\varphi)''(r_i)/2$. That is, the image of $D_{r_i}$ is contained in $y + p^2\ZZ_p$ for all $i$. Therefore, the union of the disks $D_{r_i}$ contributes at most $O(1/p^2)$ to the total image measure.

    We can say more about the coefficients $c_i$. In particular, since $\varphi(r_i) = 0$, we can write
    \[\varphi(x) = a_i(x- r_i) + O((x-r_i)^2),\]
    where $a_i = \varphi'(r_i)$. Furthermore, since $\psi'(0) = 0$, and $\psi''(0)= -1/8$, the expansion of $\psi(x)$ around $0$ is
    \[\psi(x) = -\frac{7}{4} - \frac{1}{16}x^2 + O(x^3).\]
    Using this, the expansion of $\psi\circ \varphi$ around $r_i$ can be written
    \[\psi\circ\varphi(x) = -\frac{7}{4} - \frac{1}{16}(\varphi(x))^2 + O(\varphi(x)^3) = -\frac{7}{4} - \frac{a_i^2}{16}(x-r_i)^2 + O((x-r_i)^3).\]
    Therefore,
    \[\frac{c_i}{c_j} = \frac{a_i^2}{a_j^2} \in (\FF_p^*)^2.\]
   In other words, the $c_i$ are all contained in the same square class mod $p$. Writing $x = r_i + pu$, we obtain
   \[\psi\circ\varphi(x) = -\frac{7}{4} + p^2(c_iu^2 + ph_i(u))\]
   for some analytic $h_i:\ZZ_p\to \ZZ_p$, so modulo $p^3$ we have
   \[\psi\circ\varphi(x) = -\frac{7}{4} + p^2c_iu^2 \pmod{p^3}.\]
   Since all $c_i$ live in the same square class in $\FF_p^*$, then the sets $\{c_iu^2: u\in \FF_p\}$ coincide. As $u$ varies, $u^2$ takes exactly $(p+1)/2$ values, hence $(p+1)/2$ residue classes in $-7/4 + p^2\ZZ_p$. From the expansion of $\psi\circ\varphi(x)$, the distinct values of $c_iu^2 \pmod p$ correspond to distinct classes mod $p^3$, each contributing a full ball of measure $p^{-3}$. Therefore, 
   \[\mu(\psi\circ\varphi(D_{r_1})) = \frac{p+1}{2p^3}.\]
   Since the images $\psi\circ \varphi(D_{r_i})$ coincide modulo $p^3$, then
   \[\mu(\psi\circ\varphi(D_{r_1})) = \mu\left(\bigcup_{i = 1}^3 \psi\circ \varphi(D_{r_i})\right).\]

    The case for the conjugacy class of $\infty$ is a similar calculation, since there are three preimages, each with ramification index $2$.
    
    Let $\omega_1, \omega_2$ be the roots of $x^2 + x+ 1$. Assume that neither $\omega_i$ is in a collision pair with an unramified point, as otherwise there would be no obstruction to lifting the residue class of the image. Unlike the cubic critical points above, $\omega_1$ and $\omega_2$ have ramification index $3$, and their images under $\psi\circ\varphi$ are not necessarily identified. Consequently, each disk must be considered separately, and the local expansion has the form
    \[\psi\circ\varphi(x) = y + c_i(x - \omega_i)^3 + O((x-\omega_i)^4).\]
    This leads to a dependence on the image of $u\mapsto u^3$ over $\FF_p$. The number of cubic residues (including $0$) mod $p$ is
    \[\frac{p-1}{\gcd(3,p-1)} + 1.\]
    Therefore, the image consists of $\#\text{Im}(u\mapsto u^3)$ residue classes mod $p^4$, each with measure $p^{-4}$, so
    \[\mu(\psi\circ\varphi(D_{\omega_i})) = \frac{1}{p^4}\left(\frac{p-1}{\gcd(3,p-1)} + 1\right).\]
    Suppose that $S_p$ is the size of the image of $\psi\circ\varphi$ mod $p$.  Then, the number of points whose image lifts without obstruction is $S_p - D$, if $\omega_1,\omega_2$ are not defined over $\FF_p$ (i.e. when $p \equiv 2\pmod 3$), and $S_p - D - 2$ if they are defined. When $p \equiv 2 \pmod 3$, the image density is therefore 
    \[\frac{S_p - D}{p+1} + \frac{D(p+1)}{2p^3}.\]
    When $p \equiv 1 \pmod 3$, the critical points $\omega_1,\omega_2$ are defined, and so we include altered density from the disks $D_{\omega_i}$. That is, the density is \[\frac{S_p - D -2}{p+1} + \frac{D(p+1)}{2p^3} + \frac{2(p+2)}{3p^4}.\qedhere\]   
\end{proof}
\begin{remark}\label{rmk:geometry}
    When the characteristic of the underlying field is not $7$, the equation \eqref{eq:collisioncurve} describes a singular affine plane model of a genus $4$ curve, which possesses interesting geometry in its own right\footnote{When the characteristic of the field is $7$, this curve has genus $1$.}. In particular, over a cubic extension of $\QQ$, its Jacobian variety is isogenous to a product of elliptic curves. Morton \cite{Morton92} uses this Jacobian splitting to show that the modular curve $X(3,3)$ has no $\QQ$-rational points. In other words, there are no quadratic polynomials conjugate to $x^2 + c$ with $c\in \QQ$ that have two independent rational $3$-cycles.  

    Since $N$ from Proposition \ref{prop:Fpdensity} is the number of rational points on a genus $4$ curve over $\FF_p$, we can use the Hasse-Weil bound in conjunction with the statement of Proposition \ref{prop:Fpdensity} to approximate the density of $c$ values with a $\QQ_p$-rational $3$-cycle. Applying the bound, we have 
\[|N-(p+1)| \leq 8 \sqrt{p}.\]
Furthermore, we have the bounds $1\leq K\leq 2$ and $0\leq B\leq 2$. Using these inequalities, the error term can be approximated by 
\[\frac{p+1}{18} -\frac{8\sqrt{p}}{18} - \frac{7}{3} \leq \frac{N - 9K - 12B}{18} \leq \frac{p+1}{18} +\frac{8\sqrt{p}}{18} - \frac{1}{2}.\]
Then, the expected image size of $\psi\circ \varphi$ is
\[I - \frac{N - 9K - 12B}{18} = \frac{5p}{18} + O(\sqrt{p}).\]
As $p$ gets large, the density of points in the image is $5/18$, the same value predicted in Example \ref{ex:period3density}.
\end{remark}

\bibliographystyle{plain}

\end{document}